\documentclass[11pt]{article}
\makeatletter

\usepackage{amssymb,amsmath,amsthm,latexsym,color}
\usepackage{amsfonts}\usepackage{lscape}

\newtheorem{theorem}{Theorem}[section]

\newtheorem{lemma}[theorem]{Lemma}

\newtheorem{remark}[theorem]{Remark}
\newtheorem{cor}[theorem]{Corollary}
\newtheorem{prop}[theorem]{Proposition}

\makeatother

\def\cP{{\mathcal P}}

\def\cN{\mathcal N}

\def\cBH{\mathcal {BHB}}

\def\cP{\mathcal {LMPTB}}

\def\F{\mathbb F}
\def\S{\mathbb S}
\def\V{\mathbb V}

\def\S{\mathbb S}

\def\ps@headings{
 \def\@oddhead{\footnotesize\rm\hfill\runningheadodd\hfill\thepage}
 \def\@evenhead{\footnotesize\rm\thepage\hfill\runningheadeven\hfill}
 \def\@oddfoot{}
 \def\@evenfoot{\@oddfoot}
}

\title{On isotopisms and strong isotopisms of commutative presemifields}
\author{G. Marino\thanks{This work was
supported by the Research Project of MIUR (Italian Office for
University and Research) ``Geometrie su Campi di Galois, piani di
traslazione e geometrie di incidenza''.} \quad and \ \ O. Polverino$^*$}
\date{}
\begin{document}
\maketitle
\begin{abstract}
In this paper we prove that the $P(q,\ell)$ ($q$ odd prime power and
$\ell>1$ odd) commutative semifields constructed by Bierbrauer in
\cite{BierbrauerSub} are isotopic to some commutative presemifields
constructed by Budaghyan and Helleseth in \cite{BuHe2008}. Also, we
show that they are strongly isotopic if and only if $q\equiv
1(mod\,4)$. Consequently, for each $q\equiv -1(mod\,4)$ there exist
isotopic commutative presemifields of order $q^{2\ell}$ ($\ell>1$
odd) defining CCZ--inequivalent planar DO polynomials.
\end{abstract}

\section{Introduction}

A finite \textit{semifield} $\mathbb{S}$ is a finite binary
algebraic structure satisfying all the axioms for a skewfield except
(possibly) associativity of multiplication. If $\mathbb{S}$
satisfies all axioms for a semifield except the existence of an
identity element for the multiplication, then we call it a
\textit{presemifield}. The additive group of a presemifield is an
elementary abelian $p$--group, for some prime $p$ called the {\it
characteristic} of $\S$.

The definition of nuclei and center of a semifield can be found,
for instance, in \cite[Sec. 5.9]{JoJhBi2007}.  A finite semifield is a
vector space over its nuclei and its center. Two presemifields,
say $\S_1=(\S_1,+,\bullet)$ and $\S_2=(\S_2,+,\star)$ of characteristic $p$, are said to
be \textit{isotopic} if there exist three $\F_p$--linear
permutations $M,N,L$ from $\S_1$ to $\S_2$ such that
$$
M(x) \star N(y)=L(x \bullet y)
$$
for all $x,y \in \S_1$. The triple $(M,N,L)$ is an {\it isotopism}
between $\S_1$ and $\S_2$. They are {\em strongly isotopic} if we
can choose $M=N$. From any presemifield, one can naturally construct
a semifield which is isotopic to it (see~\cite{Knuth1965}). The
sizes of the nuclei as well as the size of the center of a semifield
are invariant under isotopy. The isotopism relation between
semifields arises from the isomorphism relation between the
projective planes coordinatized by them ({\it semifield planes}).
For a recent overview on the theory of finite semifields see Chapter
\cite{LaPo201*} in the collected work \cite{deBeSt201*}.
\bigskip

\bigskip

\medskip

Commutative presemifields in odd characteristic can be equivalently
described by planar DO polynomials \cite{CoMa1997}. A {\em
Dembowski--Ostrom (DO) polynomial} $f\in\F_q[x]$ ($q=p^e$) is a
polynomial of the shape $f(x)=\sum_{i,j=0}^{e-1}a_{ij}x^{p^i+p^j}$,
whereas a polynomial $f\in\F_q[x]$ is {\em planar} or {\em perfect
nonlinear} (PN for short) if, for each $a\in\F_q^*$, the mapping
$x\mapsto f(x+a)-f(x)-f(a)$ is bijective. If $f(x)\in\F_q[x]$ is a
planar DO polynomial, then $\S_f=(\F_q,+,\star)$ is a commutative
presemifield where $x\star y=f(x+y)-f(x)-f(y)$. Conversely, if
$\S=(\F_q,+,\star)$ is a commutative presemifield of odd order, then
the polynomial $f(x)=\frac 12(x\star x)$ is a planar DO polynomial
and $\S=\S_f$.

Two functions $F$ and $F'$ from $\F_{p^n}$ to itself are called {\it
Carlet--Charpin--Zinoviev equivalent} ({\it CCZ--equivalent}) if for
some affine permutation $\cal L$ of $\F_{p^n}^2$ the image of the
graph of $F$ is the graph of $F'$, that is, ${\cal L}(G_F) = G_{F'}$
where $G_F = \{(x, F(x)) | x\in\F_{p^n}\}$ and $G_{F'} = \{(x,
F'(x)) | x\in\F_{p^n}\}$ (see \cite{CaChZi1998}). By \cite[Sec.
4]{BuHe2008}, two planar DO polynomials are CCZ--equivalent if and
only if the corresponding presemifields are strongly isotopic. In
\cite{CoHe2008}, it has been proven that two presemifields of order
$p^n$, with $p$ prime and $n$ odd integer, are strongly isotopic if
and only if they are isotopic. Whereas, for $n=6$ and $p=3$, Zhou in
\cite{ZhouSub}, by using MAGMA computations,  has shown that the
presemifields constructed in \cite{LuMaPoTr2011} and \cite{BuHe2008}
are isotopic but not strongly isotopic. In \cite{BierbrauerSub}, the author proved that the
two families of commutative presemifields constructed in \cite{BuHe2008}  are
contained, up to isotopy, into a unique family of presemifields, and
we refer to it as the family ${\cBH}$. Also in \cite{BierbrauerSub},
the author generalized the commutative semifields constructed in \cite{LuMaPoTr2011} ($\cP$ semifields) proving that
each $\cP$ semifield is not isotopic to any previously known
semifield with the possible exception of $\cBH$ presemifields.

\medskip

In  this paper we study the isotopy and strong isotopy relations
involving the above commutative presemifields, proving that the
${\cal LMPTB}$ semifields  are contained, up to isotopy, in the
family of ${\cal BHB}$ presemifields. Precisely, we show  that an
$\cP$ semifield of order $q^{2\ell}$ ($q$ odd and $\ell>1$ odd) is
isotopic to a $\cBH$ presemifield, and that they are strongly
isotopic if and if $q\equiv 1(mod\,4)$. This yields that, for planar
DO functions from $\F_{q^{2\ell}}$ to itself, when $q\equiv
-1(mod\,4)$ and $\ell>1$ odd, the isotopy relation is strictly more
general than CCZ--equivalence.

\section{Preliminary results}

If $\S=(\S,+,\bullet)$ is a presemifield, then
$\S^*=(\S,+,\bullet^*)$, where $x\bullet^* y=y\bullet x$ is a
presemifield as well, and it is called the {\em dual} of $\S$. If
$\S$ be a presemifield  of order $p^n$, then we may assume that
$\S=(\F_{p^n},+,\bullet)$, where $x\bullet
y=F(x,y)=\sum_{i,j=0}^{n-1}a_{ij}x^{p^i}y^{p^j}$, with
$a_{ij}\in\F_{p^n}$. The set $$S=\{\varphi_y:\ x\in\F_{p^n}\mapsto
F(x,y)\in\F_{p^n}\ |\ y\in\F_{p^n}\}\subseteq
\V=End(\F_{p^n},\F_p)$$ is the spread set associated with $\S$ and
$$S^*=\{\varphi^x:\ y\in\F_{p^n}\mapsto F(x,y)\in\F_{p^n}\ |\ x\in\F_{p^n}\}\subseteq \V=End(\F_{p^n},\F_p)$$ is the spread set
associated with $\S^*$. Both $S$ and $S^*$ are subgroups of order
$p^n$ of the additive group of $\V$ and each nonzero element of
$S$ and $S^*$ is invertible.

\bigskip

For each $x\in\F_{p^n}$, the {\it conjugate} $\bar\varphi$ of the
element $\varphi(x)=\sum_{i=0}^{n-1}\beta_ix^{p^i}$ of $\V$ is
defined by
$\bar\varphi(x)=\sum_{i=0}^{n-1}\beta_i^{p^{n-i}}x^{p^{n-i}}.$ The
map
$$
T:\varphi\in\V\mapsto\bar\varphi\in V
$$
is an $\F_p$--linear permutation of $\V$. Straightforward
computations show that
\begin{equation}\label{form:mapT}
\overline{\varphi\circ\psi}=\overline{\psi}\circ \overline{\varphi}\quad\quad
\overline{\varphi^{-1}}=(\overline{\varphi})^{-1}.
\end{equation}
The algebraic structure $\S^t=(\F_{p^n},+,\bullet^t)$, where
$x\bullet^t y=\overline{\varphi_y}(x)$, is a presemifield and it is
called the {\em transpose} of $\S$ (see e.g. \cite[Lemma
2]{LuMaPoTr2011}). The set $S^t=\{\overline{\varphi_y}|\
y\in\F_{p^n}\}$ is the spread set associated with $\S^t$.

\bigskip

In what follows we want to point out the relationship between
spread sets associated with two isotopic presemifields.

\begin{prop}\label{prop:IsotopicSemifAndSpreadSets}
Let $\S_1=(\F_{p^n},+,\bullet)$ and  $\S_2=(\F_{p^n},+,\star)$ be
two presemifields and let $S_1=\{\varphi_y:x\mapsto x\bullet y |\
y\in\F_{p^n}\}$ and $S_2=\{\varphi'_y:x\mapsto x\star y |\
y\in\F_{p^n}\}$ be the corresponding spread sets. Then $\S_1$ and
$\S_2$ are isotopic under the isotopism $(M,N,L)$ if and only if
$S_2=LS_1M^{-1}=\{L\circ \varphi_{y}\circ M^{-1}|\ y\in\F_{p^n}\}$.
\end{prop}
\begin{proof}
The necessary condition can be easily proven. Indeed if $(M,N,L)$
is an isotopism between $\S_1$ and $\S_2$, then
$L(\varphi_y(x))=\varphi'_{N(y)}(M(x))$ for each $x,y\in\F_{p^n}$.
Hence, $\varphi'_{N(y)}=L\circ \varphi_y\circ M^{-1}$ for each
$y\in\F_{p^n}$ and the statement follows taking
into account that $S_2=\{\varphi'_{N(y)}|\ y\in\F_{p^n}\}$.

Conversely, let $S_2=\{L\circ \varphi_{y}\circ M^{-1}|\
y\in\F_{p^n}\}$, where $M$ and $L$ are two $\F_p$--linear
permutations of $\F_{p^n}$. It is easy to see that the map $N$,
sending each element $y\in\F_{p^n}$ to the unique element
$z\in\F_{p^n}$ such that $\varphi'_{z}=L\circ \varphi_y\circ M^{-1}$
(where $\varphi_z'\in S_2$), is an $\F_p$--linear permutations of
$\F_{p^n}$. Hence, for each $x,y\in\F_{p^n}$ we get
$\varphi'_{N(y)}(x)=L(\varphi_{y}(M^{-1}(x)))$, i.e. $x\star
N(y)=L(M^{-1}(x)\bullet y)$ and putting $x'=M^{-1}(x)$ we have the
assertion.
\end{proof}

\bigskip

Let $\S=(\F_{p^n},+,\star)$ be a presemifield, where $x\star
y=F(x,y)=\sum_{i,j=0}^{n-1}a_{ij}x^{p^i}y^{p^j}$, with
$a_{ij}\in\F_{p^n}$, and let $S$ and $S^*$ be the spread sets
associated with $\S$ and $\S^*$, respectively.

The middle (respectively, right) nucleus of each semifield isotopic
to $\S$ is isomorphic to the largest field $\cN_m(\S)$
(respectively, $\cN_r(\S)$) contained in $\V=End(\F_{p^n},\F_p)$
such that $S\cN_m(\S)\subseteq S$ \footnote{By juxtaposition we will
always denote the composition of maps that will be read from right
to left.} (respectively, $\cN_r(S)S\subseteq S$), whereas the left
nucleus of each semifield isotopic to $\S$ is isomorphic to the
largest field $\cN_l(\S)$ contained in $\V$ such that
$\cN_l(\S)S^*\subseteq S^*$ (see \cite[Thm. 2.1]{MaPoTr201*} and
\cite{MaPoPrep}).

\medskip

Also, if $\F_q$ is a subfield of $\F_{p^n}$ and $F(x,y)$ is a
$q$--polynomial with respect to the variable $x$, i.e. $S\subset End
(\F_{p^n},\F_q)$, then $F_q=\{t_\lambda:\
x\in\F_{p^n}\mapsto\lambda x\in\F_{p^n}|\lambda\in\F_q\}\subset
\cN_l(\S)$ (\cite{MaPoPrep}).

If $(M,N,L)$ is an isotopism between two presemifields $\S_1$ and
$\S_2$, we have that $\cN_r(S_2)=L\cN_r(S_1)L^{-1}$,
$\cN_m(S_2)=M\cN_m(S_1)M^{-1}$ and $\cN_l(S_2)=L\cN_l(S_1)L^{-1}$
(see e.g. \cite{JoMaPoTr2008} and \cite{MaPoPrep}).

\medskip

From these results we can prove

\begin{theorem}\label{thm:semilinearity}
If $(M,N,L)$ is an isotopism between two presemifields $\S_1$ and
$\S_2$ of order $p^n$, whose associated spread sets $S_1$ and
$S_2$ are contained in $End(\F_{p^n},\F_q)$ ($\F_q$ a subfield of
$\F_{p^n}$), then $L$ and $M$ are $\F_q$--semilinear maps of
$\F_{p^n}$ with the same companion automorphism.
\end{theorem}
\begin{proof}
Since $S_1,\, S_2\subset End(\F_{p^n},\F_q)$, by the previous
arguments we have that
$$F_q=\{t_\lambda:\ x\in\F_{p^n}\mapsto\lambda
x\in\F_{p^n}|\lambda\in\F_q\}\subset \cN_l(\S_1)\cap\cN_l(\S_2).$$
Also $\cN_l(S_2)=L\cN_l(S_1)L^{-1}$. Then $L^{-1}F_qL\subset
\cN_l(\S_2)$, and since a field contains a unique subfield of given
order, it follows $L^{-1}F_qL=F_q$. Since the map
$t_\lambda\mapsto L^{-1}t_\lambda L$ is an automorphism
of the field of maps $F_q$, there exists $i\in\{0,\dots,n-1\}$ such
that $L^{-1}t_\lambda L=t_{\lambda^{p^i}}$ for each
$\lambda\in\F_q$, i.e. $L$ is an $\F_q$--semilinear map of
$\F_{p^n}$ with companion automorphism $\sigma(x)=x^{p^i}$. Also, by
Proposition \ref{prop:IsotopicSemifAndSpreadSets}, $LS_1M^{-1}=S_2$,
and hence $M$ is an $\F_q$--semilinear map of $\F_{p^n}$ as well,
with the same companion automorphism $\sigma$.
\end{proof}

\bigskip

Finally, since the dual and the transpose operations are invariant
under isotopy \cite{Knuth1965}, it makes sense to ask which is the
isotopism involving the duals and the transposes of two isotopic
presemifields. We have the following result.

\begin{prop}\label{prop:DualTranspIsot}
Let $\S_1$ and $\S_2$ be two presemifields. Then
\begin{itemize}
\item[$i)$] $(M,N,L)$ is an isotopism between $\S_1$ and $\S_2$ if
and only if $(N,M,L)$ is an isotopism between the dual
presemifields $\S_1^*$ and $\S_2^*$; \item[$ii)$]  $(M,N,L)$ is an
isotopism between $\S_1$ and $\S_2$ if and only if
$(\overline{L}^{-1},N,\overline{M^{-1}})$ is an isotopism between
the transpose presemifields $\S_1^t$ and $\S_2^t$; \item[$iii)$]
$(M,N,L)$ is an isotopism between $\S_1$ and $\S_2$ if and only
$(N,\overline{L}^{\,-1},\overline{M}^{\,-1})$ is an isotopism
between  $\S_1^{t*}$ and $\S_2^{t*}$.
\end{itemize}
\end{prop}
\begin{proof} Statement $i)$ easily follows from the definition of the dual
operation, whereas $iii)$ follows from $i)$ and $ii)$.

Let us prove $ii)$. Let $\S_1=(\F_{p^n},+,\bullet)$ and
$\S_2=(\F_{p^n},+,\star)$ and let $S_1=\{\varphi_y|\ y\in\F_{p^n}\}$
and $S_2=\{\varphi'_y|\ y\in\F_{p^n}\}$ be the corresponding spread
sets. By the previous arguments the corresponding transpose
presemifields are $\S_1^t=(\F_{p^n},+,\bullet^t)$ and
$\S_2^t=(\F_{p^n},+,\star^t)$, where $x\bullet^t
y=\overline{\varphi_y}(x)$ and $x\star^t
y=\overline{\varphi'_y}(x)$, respectively. The triple $(M,N,L)$ is
an isotopism between $\S_1$ and $\S_2$ if and only if
$L\circ\varphi_y=\varphi'_{N(y)}\circ M$ for each $y\in\F_{p^n}$. By
(\ref{form:mapT}), $\overline{\varphi_y}\circ
\overline{L}=\overline{M}\circ \overline{\varphi'_{N(y)}}$ for each
$y\in\F_{p^n}$ and hence $$\overline{L}(x)\bullet^t
y=\overline{M}(x\star^tN(y))$$ for each $x,y\in\F_{p^n}$. By
(\ref{form:mapT}), this is equivalent to
$\overline{M^{-1}}(z\bullet^t y)=\overline{L^{-1}}(z)\star^tN(y)$
for each $z,y\in\F_{p^n}$. The assertion follows.
\end{proof}

Finally, by $iii)$ of Proposition \ref{prop:DualTranspIsot} and by
Proposition \ref{prop:IsotopicSemifAndSpreadSets} we immediately
get the following result.

\begin{cor}\label{cor:IsotopyComSimpl-spreads}
Let $\S_1=(\F_{p^n},+,\bullet)$ and $\S_2=(\F_{p^n},+,\star)$ be two
presemifields and let $S_1^{t*}$ and $S_2^{t*}$ be the spread sets
associated with the presemifields $\S_1^{t*}$ and $\S_2^{t*}$,
respectively. Then $\S_1$ and $\S_2$ are strongly isotopic if and
only if there exists an $\F_p$--linear permutation $H$ of $\F_{p^n}$
such that $S_2^{t*}=HS_1^{t*}\overline{H}$.
\end{cor}

\section{$\cBH$ and $\cP$ commutative presemifields}
The $\cBH$ presemifields and the $\cP$ semifields  presented in
\cite{BierbrauerSub} can be described as follows.

\bigskip

 $\cBH$)\quad\quad {\em $B(p,m,s,\beta)$ presemifields}
{\rm\cite{BuHe2008}, \cite{BierbrauerSub}}:\quad
$(\F_{p^{2m}},+,\star)$, $p$ odd prime and $m>1$, with
\begin{equation}\label{form:multiplSimplBHp}
x\star
y=xy^{p^m}+x^{p^m}y+[\beta(xy^{p^s}+x^{p^s}y)+\beta^{p^m}(xy^{p^s}+x^{p^s}y)^{p^m}]\omega,
\end{equation}
where $0<s<2m$, $\omega$ is an element of $\F_{p^{2m}}\setminus
\F_{p^m}$ with $\omega^{p^m}=-\omega$ and the following conditions
are satisfied:
\begin{equation}\label{form:condMultSimplBHp}
\beta\in\F_{p^{2m}}^*:\ \ \beta^{\frac{p^{2m}-1}{(p^m+1,p^s+1)}}\ne
1\quad\quad\mbox{and}\quad\quad \mbox{$\not\!\exists\
a\in\F_{p^{2m}}^*:$ $a+a^{p^m}=a+a^{p^s}=0$.}
\end{equation}

\bigskip

$\cP$)\quad{\em $P(q,\ell)$ semifields} {\cite{LuMaPoTr2011},
\cite{BierbrauerSub}}:\quad $(\F_{q^{2\ell}},+,*)$, $q$ odd prime
power and $\ell=2k+1>1$ odd, with
$$x*y=\frac{1}2(xy+x^{q^\ell}y^{q^\ell})+\frac 14G(xy^{q^2}+x^{q^2}y),
$$
where
$G(x)=\sum_{i=1}^{k}(-1)^{i}(x-x^{q^{\ell}})^{q^{2i}}+\sum_{j=1}^{k-1}(-1)^{k+j}(x-x^{q^{\ell}})^{q^{2j+1}}$.

\bigskip

In order to prove our results, we start by further investigating
Multiplication (\ref{form:multiplSimplBHp}) and Conditions
(\ref{form:condMultSimplBHp}). Set $h:=\gcd(m,s)$, then $m=h\ell$
and $s=hd$, where $\ell$ and $d$ are two positive integers such that
$0<d<2\ell$ and $\gcd(\ell,d)=1$. Putting $q=p^h$, then
$\omega\in\F_{q^{2\ell}}\setminus \F_{q^\ell}$ such that
$\omega^{q^\ell}=-\omega$ and the $\cBH$ presemifields
$B(p,m,s,\beta)=(\F_{q^{2\ell}},+,\star)$ will be denoted by
$\overline{B}(q,\ell,d,\beta)$. Moreover, Multiplication
(\ref{form:multiplSimplBHp}) and Conditions
(\ref{form:condMultSimplBHp}) can be rewritten as
$$
x\star y=
xy^{q^\ell}+x^{q^\ell}y+[\beta(xy^{q^d}+x^{q^d}y)+\beta^{q^\ell}(xy^{q^d}+x^{q^d}y)^{q^\ell}]\omega,
$$
where
\begin{equation}\label{form-beta1}
\beta\in\F_{q^{2\ell}}^*:\ \
\beta^{\frac{q^{2\ell}-1}{(q^\ell+1,q^d+1)}}\ne 1,
\end{equation}
and
\begin{equation}\label{form-s1}
\not\!\exists\ a\in\F_{q^{2\ell}}^*:\ \ a+a^{q^\ell}=a+a^{q^d}=0.
\end{equation}
\bigskip

We get the following preliminary result.

\begin{lemma}\label{lemma:beta&s}

\begin{itemize}
\item[i)] Condition (\ref{form-s1}) is fulfilled if and only if
$\ell+d$ is odd. \item[ii)] If Condition (\ref{form-s1}) is
fulfilled, then an element $\beta\in\F_{q^{2\ell}}^*$ satisfies
Condition (\ref{form-beta1}) if and only if $\beta$ is a nonsquare
of $\F_{q^{2\ell}}$.
\end{itemize}
\end{lemma}
\begin{proof}
$i)$\quad The sufficient condition can be easily proven. Indeed,
since $\gcd(\ell,d)=1$ then $\ell$ and $d$ cannot be both even
integers. Moreover, if $\ell$ and $d$ were both odd, then each
element $a\in\F_{q^2}$ such that $a^q=-a$ would be a solution of
$x^{q^\ell}=x^{q^d}=-x$, contradicting our assumption. On the
other hand, suppose that $\ell+d$ is odd, then
$\gcd(2\ell,\ell+d)=\gcd(\ell,d)=1$. Hence, if there exists an
element $a\in\F_{q^{2\ell}}^*$ such that
$a^{q^\ell}+a=a^{q^d}+a=0$, then $a$ satisfies the equation
$x^{q^{\ell+d}-1}=1$, which admits
$\gcd(q^{2\ell}-1,q^{\ell+d}-1)=q^{\gcd(2\ell,\ell+d)}-1=q-1$
solutions. It follows that $a\in\F_q^*$, a contradiction.

\smallskip
$ii)$\quad We first suppose $\ell$ is odd and $d$ is even and
prove that $\gcd(q^\ell+1,q^d+1)=2$. If $q\equiv 1\,(mod\,4)$,
then $q^\ell+1\equiv q^d+1\equiv 2\,(mod\ 4)$. On the other hand,
if $q\equiv 3\,(mod\,4)$, since $\ell$ is odd and $d$ is even,
$q^\ell+1\equiv 0\,(mod\ 4)$ and $q^d+1\equiv 2\,(mod\ 4)$. So in
both cases $2$ is the maximum power of 2 dividing
$\gcd(q^\ell+1,q^d+1)$. Now suppose that $p'$ is an odd prime such
that $p'|(q^\ell+1)$ and $p'|(q^d+1)$. Hence $q^\ell\equiv
-1\,(mod\ p') $ and $q^d\equiv -1\,(mod\ p')$. Since
$\gcd(\ell,d)=1$, then $1=a\ell+bd$, with $a$ an odd integer. From
the previous congruences it follows that $q=q^{a\ell+bd}\equiv
(-1)^a(-1)^b\,(mod\ p')\equiv (-1)^{b+1}\,(mod\ p')$ and since $d$
is even, we have $q^d\equiv 1\,(mod\ p')$, a contradiction.

If $\ell$ is even and $d$ is odd, arguing as in the previous case
we obtain the assertion.
\end{proof}

\begin{remark}\label{rm:BHconditions}
{\rm By Lemma \ref{lemma:beta&s}, the algebraic structure $\overline{B}(q,\ell,d,\beta)$
is a presemifield if and only if {\bf $\ell+d$ is odd and $\beta$
is a nonsquare in $\F_{q^{2\ell}}$.}}
\end{remark}

In \cite{BierbrauerSub}, the author proved that the semifields
$P(q,\ell)$ are not isotopic to any previously known commutative semifield
with the possible exception of $\cBH$ presemifields. In what
follows, using the notation introduced in this section, we study
the isotopy relation involving the families of presemifields
$P(q,\ell)$ and $\overline{B}(q,\ell,d,\beta)$ and we prove that a $P(q,\ell)$ semifield of order $q^{2\ell}$, with $q=p^e$ an
odd prime power and $\ell>1$ an odd integer, is isotopic to a
$\overline{B}(q,\ell,2,\beta)$ presemifield for a suitable choice
of $\beta$.

\section{The isotopism issue}
By \cite{Kantor2003}, there is a canonical bijection between
commutative and symplectic presemifields. Precisely, if $\S$ is a
commutative presemifield, then $\S^{t*}$ is a symplectic
presemifield. Moreover, by $iii)$ of Proposition
\ref{prop:DualTranspIsot}, two commutative presemifields are
isotopic if and only if the corresponding symplectic presemifields
are isotopic as well. So, in the next, we will prove that the
symplectic presemifield $P(q,\ell)^{t*}$ is isotopic to a symplectic
presemifield $\overline{B}(q,\ell,2,\beta)^{t*}$.

\subsubsection*{The symplectic version of $P(q,\ell)$ semifields}

From \cite[Sec. 3]{BierbrauerSub}, the symplectic presemifield
arising from the commutative semifield $P(q,\ell)$, $q$ an odd
prime power and $\ell=2k+1$ an odd integer, is
$P(q,\ell)^{t*}=(\F_{q^{2\ell}},+,\bullet)$ with multiplication
given by

$$
x\bullet y=\frac{y+y^{q^\ell}}2x+\frac 14
(y-y^{q^\ell}+\alpha_y+\beta_y+\gamma_y)x^{q^2}+\frac
14(y-y^{q^\ell}-\alpha_y-\beta_y-\gamma_y)x^{q^{2\ell-2}},
$$
where $\alpha_y=\sum_{i=1}^{\ell-1}(-1)^{i+1}y^{q^{2i}}$,
$\beta_y=\sum_{j=0}^{k-1}(-1)^{k+j+1}y^{q^{2j+1}}$ and
$\gamma_y=\sum_{t=k+1}^{\ell-1}(-1)^{k+t}y^{q^{2t+1}}$.

Setting $g(y):=\alpha_y+\beta_y+\gamma_y$ and
$$
f(y):=\frac 14(y-y^{q^\ell}+g(y)),$$ direct computations show that
\begin{equation}\label{form:f}f(y)^{q^{2\ell-2}}=\frac
14(y-y^{q^\ell}-g(y)).\end{equation} Indeed, reducing modulo
$y^{q^{2\ell}}-y$, we have
$$4\,f(y)^{q^{2\ell-2}}=y^{q^{2\ell-2}}-y^{q^{\ell-2}}+\sum_{i=1}^{\ell-1}(-1)^{i+1}y^{q^{2(i-1)}}+
\sum_{j=0}^{k-1}(-1)^{k+j+1}y^{q^{2j-1}}+\sum_{t=k+1}^{\ell-1}(-1)^{k+t}y^{q^{2t-1}}$$
and setting $i'=i-1$, $j'=j-1$, $t'=t-1$, we get
\begin{eqnarray*}
4\,f(y)^{q^{2\ell-2}}&=&y-y^{q^\ell}+\sum_{i'=1}^{\ell-1}(-1)^{i'}y^{q^{2i'}}+
\sum_{j'=0}^{k-1}(-1)^{k+j'+1}y^{q^{2j'+1}}+\sum_{t'=k+1}^{\ell-1}(-1)^{k+t'}y^{q^{2t'+1}}\\
&=&y-y^{q^\ell}-(\alpha_y+\beta_y+\gamma_y).
\end{eqnarray*}
Hence
\begin{equation}\label{form:multiplSimplBr1}
x\bullet y=\frac{y+y^{q^\ell}}2x+f(y)x^{q^2}+f(y)^{q^{2\ell-2}}x^{q^{2\ell-2}}.
\end{equation}

Let $\eta\in\F_{q^2}\setminus\F_q$ such that
$\eta^q=-\eta$. Since $q$ and $\ell=2k+1$ are odd integers, the map
$\phi:\gamma\in\F_{q^\ell}\mapsto
\gamma+\gamma^{q^2}\in\F_{q^\ell}$ is invertible and
$$\phi^{-1}:z\in\F_{q^\ell}\mapsto \frac 12\Bigg(\sum_{i=0}^k(-1)^iz^{q^{2i}}+\sum_{j=0}^{k-1}(-1)^{k+j+1}z^{q^{2j+1}}\Bigg)\in\F_{q^\ell}.$$
Taking into account that $\{1,\eta\}$ is an $\F_{q^\ell}$--basis
of $\F_{q^{2\ell}}$ and that $\phi$ is an invertible map, it
follows that any element $y\in\F_{q^{2\ell}}$ can be uniquely written
as
$$y=A+(B^{q^2}+B)\eta,$$ with $A,B\in\F_{q^\ell}$.
Also
\begin{equation}\label{form:alpha}
A=\frac{y+y^{q^\ell}}{2}
\end{equation}
and
$$B^{q^2}+B=\frac{y-y^{q^\ell}}{2\eta}$$
Direct computations show that
\begin{eqnarray}\label{form-beta2}
B &=&\phi^{-1}\Bigg(\frac{y-y^{q\ell}}{2\eta}\Bigg)=
\frac 12\Bigg(\sum_{i=0}^k(-1)^i\frac{(y-y^{q^\ell})^{q^{2i}}}{2\eta}+\sum_{j=0}^{k-1}(-1)^{k+j+1}\frac{(y-y^{q^\ell})^{q^{2j+1}}}{-2\eta}\Bigg)\\
&=& \frac
1{4\eta}\Bigg(y-y^{q^\ell}+\sum_{i=1}^k(-1)^iy^{q^{2i}}-\sum_{i=1}^k(-1)^iy^{q^{\ell+2i}}-
\sum_{j=0}^{k-1}(-1)^{k+j+1}y^{q^{2j+1}}+\sum_{j=0}^{k-1}(-1)^{k+j+1}y^{q^{\ell+2j+1}}\Bigg).\nonumber
\end{eqnarray}
Putting $2t+1:=\ell+2i$, i.e. $i=t-k$, we have
$$
\sum_{i=1}^k(-1)^iy^{q^{\ell+2i}}=\sum_{t=k+1}^{\ell-1}(-1)^{t-k}y^{q^{2t+1}}=\sum_{t=k+1}^{\ell-1}(-1)^{t+k}y^{q^{2t+1}}
$$
and putting $2v:=\ell+2j+1$, i.e. $j=v-k-1$, we have
$$
\sum_{j=0}^{k-1}(-1)^{k+j+1}y^{q^{\ell+2j+1}}=\sum_{v=k+1}^{\ell-1}(-1)^vy^{q^{2v}}.
$$
Hence, substituting the last two equalities in Equation (\ref{form-beta2}), we get
$$B =\frac1{4\eta}\Bigg(y-y^{q^\ell}+\sum_{i=1}^{\ell-1}(-1)^iy^{q^{2i}}-
\sum_{j=0}^{k-1}(-1)^{k+j+1}y^{q^{2j+1}}-
\sum_{t=k+1}^{\ell-1}(-1)^{t+k}y^{q^{2t+1}}\Bigg)=\frac1{4\eta}(y-y^{q^\ell}-\alpha_y-\beta_y-\gamma_y)$$
and, taking (\ref{form:f}) into account, this yields
$f(y)=B^{q^2}\eta$. Hence, from (\ref{form:multiplSimplBr1}),
(\ref{form:alpha}) and the last equality, we get the following
result.
\begin{prop}\label{prop:MultSimplBR1}
The symplectic presemifield
$P(q,\ell)^{t*}=(\F_{q^{2\ell}},+,\bullet)$ arising from the
commutative semifield $P(q,\ell)$ has multiplication
$$x\bullet y=A x+B^{q^2}\eta x^{q^2}+B\eta x^{q^{2\ell-2}},$$
where $\eta$ is a given element of $\F_{q^2}\setminus\F_q$ with
$\eta^q=-\eta$ and $y=A+(B^{q^2}+B)\eta$, $A,B\in\F_{q^\ell}$.\qed
\end{prop}

\subsubsection*{The symplectic version of $\overline{B}(q,\ell,d,\beta)$--presemifields}

Let $q$ be an odd prime power, $\ell$ and $d$ be integers such
that $0<d<2\ell$, $\ell+d$ is odd and $\gcd(\ell,d)=1$. Then a
commutative $\overline{B}(q,\ell,d,\beta)$--presemifield is of
type $(\F_{q^{2\ell}},+,\star)$, where
$$
x\star y=
xy^{q^\ell}+x^{q^\ell}y+[\beta(xy^{q^d}+x^{q^d}y)+\beta^{q^\ell}(xy^{q^d}+x^{q^d}y)^{q^\ell}]\omega,$$
with $\beta$ a nonsquare in $\F_{q^{2\ell}}$  and
$\omega^{q^\ell}=-\omega$ (see Remark \ref{rm:BHconditions}). By using \cite[Lemmas 1,
2]{LuMaPoTr2011}, the transpose semifield
$\overline{B}^{t}(q,\ell,d,\beta)=(\F_{q^{2\ell}},+, \star^t)$ of
$\overline{B}(q,\ell,d,\beta)$ is defined by
$$x\star^t y=(x+x^{q^\ell})y^{q^\ell}+\beta^{q^{2\ell-d}}\omega^{q^{2\ell-d}}(x^{q^{2\ell-d}}-x^{q^{\ell-d}})y^{q^{2\ell-d}}+
\beta\omega(x-x^{q^{\ell}})y^{q^{d}}.$$ Hence
$\overline{B}^{\,t*}(q,\ell,d,\beta)=(\F_{q^{2\ell}},+,\star^{t*})$,
where
$$x\star^{t*} y=(y+y^{q^\ell})x^{q^\ell}+\beta^{q^{2\ell-d}}\omega^{q^{2\ell-d}}(y^{q^{2\ell-d}}-y^{q^{\ell-d}})x^{q^{2\ell-d}}+
\beta\omega(y-y^{q^{\ell}})x^{q^{d}}.$$ Since $\{1,\omega\}$
is an $\F_{q^\ell}$--basis of $\F_{q^{2\ell}}$, putting
$y=A+B\omega$, with $A,B\in\F_{q^\ell}$ and recalling that
$\omega^{q^\ell}=-\omega$, and hence
$\omega^2=\sigma\in\F_{q^\ell}^*$, we get
\begin{prop}\label{prop:MultSimplBR}
The symplectic presemifield
$\overline{B}(q,\ell,d,\beta)^{t*}=(\F_{q^{2\ell}},+,\star')$
arising from the commutative semifield
$\overline{B}(q,\ell,d,\beta)$ has multiplication
\begin{equation}\label{form:multiplSimplBr2}
x\star' y=2Ax^{q^\ell}+2\sigma^{q^{2\ell-d}}
\beta^{q^{2\ell-d}}B^{q^{2\ell-d}} x^{q^{2\ell-d}}+2\sigma\beta
Bx^{q^{d}},
\end{equation}
where $\beta$ is a nonsquare in $\F_{q^{2\ell}}$ and $y=A+B\omega$
with $A,B\in\F_{q^\ell}$, $\sigma$ is a nonsquare in $\F_{q^\ell}$
and $\omega^2=\sigma$.\qed
\end{prop}

\begin{remark}\label{rem:omega}{\rm
Note that if $\sigma$ and $\sigma'$ are two nonsquare elements of
$\F_{q^{\ell}}$ , then $\sigma'=t\sigma$, where $t$ is a nonzero
square in $\F_{q^\ell}$. So, replacing $\beta$ by $t\beta$ in
(\ref{form:multiplSimplBr2}), we may substitute $\sigma$ with
$\sigma'$. It follows that, when {\bf $\ell$ is odd}, in order to
study, up to isotopy, the $\cBH$ presemifields we may
suppose wlg that $\sigma$ is a nonsquare in $\F_{q}$ and hence
$\omega \in \F_{q^{2}}\setminus \F_q$.}
\end{remark}

\subsubsection*{The isotopism theorem}

Let start by proving the following
\begin{theorem}\label{prop:isotopia}
Let $q$ be an odd prime power, let $\ell$ and $d$ be odd and even
integers, respectively, such that $0<d<2\ell$ and
$\gcd(\ell,d)=1$. The symplectic presemifield
$\overline{B}(q,\ell,d,\beta)^{t*}=(\F_{q^{2\ell}},+,\star')$,
whose multiplication is given in (\ref{form:multiplSimplBr2}), is
isotopic to a presemifield $(\F_{q^{2\ell}},+,\star'')$ whose
multiplication is given by
$$x\star'' y=2\Bigg(Ax+\sigma
B\omega \frac{\beta}{\xi^{q^{\ell}}} x^{q^d}+\sigma
B^{q^{2\ell-d}}\omega\frac{\beta^{q^{2\ell-d}}}{\xi^{q^\ell}}x^{q^{2\ell-d}}\Bigg),$$
where $y=A+B\omega$ with $A,B\in\F_{q^\ell}$,
$\omega\in\F_{q^2}\setminus\F_q$ with $\omega^2=\sigma\in\F_q^*$,
and $\xi$ is an element of $\F_{q^{2\ell}}$ such that
$\xi^{q^{\ell+d}-1}=\beta^{1-q^\ell}$ and $\xi^{q^\ell+1}=\sigma$.
\end{theorem}
\begin{proof}
By Proposition \ref{prop:MultSimplBR} and Remark \ref{rem:omega},
the spread set associated with the symplectic presemifield
$\overline{B}(q,\ell,d,\beta)^{t*}=(\F_{q^{2\ell}},+,\star')$ is
$$S=\{\varphi_y=\varphi_{A,B}:x\mapsto
2Ax^{q^\ell}+2\sigma \beta^{q^{2\ell-d}}B^{q^{2\ell-d}}
x^{q^{2\ell-d}}+2\sigma\beta Bx^{q^{d}}\,|\ y=A+B\omega,\,A,B
\in\F_{q^\ell}\},
$$
where $\beta$ and $\sigma$ are nonsquares in $\F_{q^{2\ell}}$ and
$\F_q$, respectively.

Since $\gcd(q^{2\ell}-1,q^{\ell+d}-1)=q-1$ and
$(\beta^{1-q^\ell})^{\frac{q^{2\ell}-1}{q-1}}=1$ , the following
equation
\begin{equation}\label{form:xi}
x^{q^{\ell+d}-1}=\beta^{1-q^\ell}.\end{equation} admits $q-1$
distinct solutions in $\F_{q^{2\ell}}$. Moreover, if $\xi$ and
$\bar\xi$ satisfy (\ref{form:xi}), then $\xi/\bar\xi\in\F_q^*$.
Also, if $\xi$ is a solution of (\ref{form:xi}), then
$\xi^{q^\ell+1}$ is a solution of $x^{q^{\ell+d}-1}=1$ and since
$\gcd(q^{2\ell}-1,q^{\ell+d}-1)=q-1$, we get
$\xi^{q^\ell+1}\in\F_q^*$. Moreover, taking into account that
$\beta$ is a nonsquare in $\F_{q^{2\ell}}$, it follows that
$\xi^{q^\ell+1}$ is a nonsquare in $\F_q$. Indeed if
$(\xi^{q^\ell+1})^{\frac{q-1}2}=1$, then $(\frac
1\beta)^{\frac{q^{2\ell}-1}2}=(\xi^{q^{\ell+d}-1})^{\frac{q^{\ell}+1}2}=(\xi^{q^\ell+1})^{\frac{q^{\ell+d}-1}2}=1$,
a contradiction. Hence the set $\{\xi^{q^{\ell}+1}\,|\ \mbox{$\xi$
is a solution of (\ref{form:xi})}\}\subset \F_q$ is the set of
nonsquares in $\F_q$. This means that we can choose
$\xi\in\F_{q^{2\ell}}$, satisfying (\ref{form:xi}) and such that
\begin{equation}\label{form:xi1}
\xi^{q^\ell+1}=\sigma=\omega^2. \end{equation}

\medskip

Now, consider the invertible maps  of $\F_{q^{2\ell}}$
$$\psi:x\mapsto \frac\omega\xi x+x^{q^{\ell}} \mbox{\quad and \quad}\phi:x\mapsto
x-\frac\omega{\xi^{q^\ell}} x^{q^{\ell}}$$ and note that
$$\psi^{-1}:x\mapsto \frac
12(\frac\omega{\xi^{q^\ell}}x+x^{q^\ell}) \mbox{\quad and
\quad}\psi^{-1}( \phi(x)^{q^\ell})=x.$$ Since  $\psi$ and $\phi$
are linear maps over $\F_{q^\ell}$, for each $x\in\F_{q^{2\ell}}$
we have
\begin{eqnarray}\label{form:isotopia3}
\psi^{-1}\circ \varphi_{A,B}\circ\phi
(x)&=&2(\psi^{-1}(A(\phi(x))^{q^\ell}+\sigma
\beta^{q^{2\ell-d}}B^{q^{2\ell-d}}
(\phi(x))^{q^{2\ell-d}}+\sigma\beta
B(\phi(x))^{q^{d}}))\nonumber\\
&=& 2(Ax+\sigma B^{q^{2\ell-d}}\psi^{-1}(f(x))+\sigma
B\psi^{-1}(g(x))),
\end{eqnarray}
where $f(x)=(\beta\phi(x))^{q^{2\ell-d}}$ and
$g(x)=\beta(\phi(x))^{q^d}$.

Then, taking into account that $\omega^q=-\omega$, direct
computations show that
$$\psi^{-1}(f(x))=\frac 12 f_1x^{q^{\ell-d}}+\frac
12f_2x^{q^{2\ell-d}},$$ with
$f_1=-\frac{\omega^2}{\xi^{q^\ell+q^{\ell-d}}}\beta^{q^{2\ell-d}}+\beta^{q^{\ell-d}}$
and
$f_2=\frac\omega{\xi^{q^\ell}}\beta^{q^{2\ell-d}}+\frac{\omega}{\xi^{q^{2\ell-d}}}\beta^{q^{\ell-d}}$.

By (\ref{form:xi}), we get
$\beta^{q^\ell}=\frac{\beta\xi}{\xi^{q^{\ell+d}}}$ and elevating
to the $q^{2\ell-d}$--th power we have
$\beta^{q^{\ell-d}}=\beta^{q^{2\ell-d}}\xi^{q^{\ell}(q^{\ell-d}-1)}$.
From (\ref{form:xi1}) it follows
$\beta^{q^{\ell-d}}=\beta^{q^{2\ell-d}}(\frac{\omega^2}{\xi})^{(q^{\ell-d}-1)}=(\beta^{q^{2\ell-d}}\frac{\omega^2}{\xi^{q^{\ell-d}}})\frac\xi{\omega^2}
=\beta^{q^{2\ell-d}}\frac{\omega^2}{\xi^{q^{\ell-d}+q^\ell}}$;
hence $f_1=0$.

Also,
$f_2=\omega(\frac{\beta^{q^{\ell-d}}}{\xi^{q^{2\ell-d}}}+\frac{\beta^{q^{2\ell-d}}}{\xi^{q^\ell}})$
and by (\ref{form:xi})
 we have $f_2=2\omega
 \frac{\beta^{q^{2\ell-d}}}{\xi^{q^{\ell}}}$. Hence, $\psi^{-1}(f(x))=\omega
 \frac{\beta^{q^{2\ell-d}}}{\xi^{q^{\ell}}} x^{q^{2\ell-d}}$, and using similar arguments we have $\psi^{-1}(g(x))=\omega
 \frac{\beta}{\xi^{q^{\ell}}} x^{q^d}$. Then, by
 (\ref{form:isotopia3}), we get
 $$\psi^{-1}\circ \varphi_{A,B}\circ\phi(x)=2Ax+2\sigma B\omega \frac{\beta}{\xi^{q^{\ell}}} x^{q^d}+2\sigma
B^{q^{2\ell-d}}\omega\frac{\beta^{q^{2\ell-d}}}{\xi^{q^\ell}}x^{q^{2\ell-d}}.$$
Hence, $$\psi^{-1}\circ \varphi_y\circ\phi(x)=x\star'' y,$$ i.e.
\begin{equation}\label{form:isotopism1}
\phi(x)\star'y=\psi(x\star'' y).
\end{equation}
This means that $(\phi,id,\psi)$ is an isotopism between the two
presemifields.  The theorem is proven.
\end{proof}

\begin{theorem}\label{thm:isotopyPqlBH}
Each $\cP$ semifield is isotopic to a $\cBH$ presemifield.
\end{theorem}
\begin{proof}
By Proposition \ref{prop:MultSimplBR1} the symplectic presemifield
$P(q,\ell)^{t*}=(\F_{q^{2\ell}},+,\bullet)$, $q$ odd and $\ell>1$
odd, arising from the commutative semifield $P(q,\ell)$ has
multiplication
$$
x\bullet y=A x+B^{q^2}\eta x^{q^2}+B\eta x^{q^{2\ell-2}},
$$
where $\eta^q=-\eta$ and $y=A+(B^{q^2}+B)\eta$ with
$A,B\in\F_{q^\ell}$.

Put $d=2$ in Theorem \ref{prop:isotopia} and choose
$\beta=\bar\beta$ as a nonsquare in $\F_{q^{2\ell}}$ belonging to
$\F_{q^2}$ such that ${\bar\beta}^{q+1}=\frac 1\sigma$. Then
${\bar\beta}^{\,-1}$ is a solution of Equation (\ref{form:xi}) and
since ${\bar\beta}^{q^\ell+1}={\bar\beta}^{q+1}=\frac 1\sigma$, we
can fix $\xi={\bar\beta}^{\,-1}$. By Theorem \ref{prop:isotopia} the
symplectic presemifield $\overline{B}(q,\ell,2,\bar\beta)^{t*}$ is
isotopic to the presemifield $(\F_{q^{2\ell}},+,\star'')$ whose
multiplication is given by
$$x\star'' y=2Ax+2B\omega x^{q^2}+
2B^{q^{2\ell-2}}\omega x^{q^{2\ell-2}},$$ where $\omega^q=-\omega$
and $y=A+B\omega$ with $A,B\in\F_{q^\ell}$. Let $\omega=\alpha
\eta$ and note that $\alpha \in \F_q^*$.

Let $h:\, y=A+B\omega\in\F_{q^{2\ell}}\mapsto\
2A+2(B^{q^{2\ell-2}}+B)\omega\in\F_{q^{2\ell}}$. Since $q$ and
$\ell$ are odd,  $h$ is an invertible $\F_q$--linear map of
$\F_{q^{2\ell}}$. Also, since $h(y)=h(A+B\omega)= 2A+2((\alpha
B^{q^{2\ell-2}})^{q^2}+(\alpha B^{q^{2\ell-2}}))\eta $ we have
$$x\bullet h(y)=x\star''y$$
for each $x,y\in\F_{q^\ell}$, hence by (\ref{form:isotopism1}) we
get
$$\phi(x)\star'h^{-1}(z)=\psi(x\bullet z)$$ for each $x,z\in\F_{q^\ell}$.
Then $(\phi,h^{-1},\psi)$ is an isotopism between $P(q,\ell)^{t*}$
and $\overline{B}(q,\ell,2,\bar\beta)^{t*}$. The theorem is proven.
\end{proof}

By Theorems \ref{prop:isotopia}, \ref{thm:isotopyPqlBH} and by
$iii)$ of Proposition \ref{prop:DualTranspIsot} we can state the
following
  result.

\begin{cor}\label{cor:isotopy}
The triple $(\bar\psi^{-1},\phi,\bar h)$ is an isotopism between
the commutative semifield $P(q,\ell)$ and the presemifield
$\overline{B}(q,\ell,2,\bar\beta)$, where $\bar\beta$ is a
nonsquare in $\F_{q^2}$.
\end{cor}

\begin{remark}
{\rm Note that, since $\bar\psi^{-1} \neq \phi$, the above isotopism
is not a strong isotopism.}
\end{remark}

\section{Strong Isotopism}
In this section we will prove that the isotopic presemifields
$P(q,\ell)$ and $\overline{B}(q,\ell,2,\bar\beta)$ of Corollary
\ref{cor:isotopy}, are strongly isotopic if and only if $q\equiv
1(mod\,4)$. Let us start by proving the following.

\begin{theorem}\label{thm:strongisotopy}
If $q\equiv 1(mod\,4)$, then the commutative presemifields
$P(q,\ell)$ and $\overline{B}(q,\ell,2,\bar\beta)$ of Corollary
\ref{cor:isotopy} are strongly istopic.
\end{theorem}

\begin{proof}
By Corollary \ref{cor:IsotopyComSimpl-spreads}, the two involved
presemifields are strongly isotopic if and only if there exists an
invertible $\F_p$--linear map $H$ of $\F_{q^{2\ell}}$, such that
$HS_1\overline H=S_2$, where $S_1$ and $S_2$ are the spread sets
associated with $P(q,\ell)^{t*}$ and
$\overline{B}(q,\ell,2,\bar\beta)^{t*}$, respectively. By the proof
of Theorem \ref{thm:isotopyPqlBH} and by Proposition
\ref{prop:IsotopicSemifAndSpreadSets}, we have that $\psi
S_1\phi^{-1}=S_2$, where $$\psi:x\mapsto \omega\bar\beta
x+x^{q^{\ell}} \mbox{\quad and \quad}\phi^{-1}:x\mapsto \frac
12(x+\omega\bar\beta^q x^{q^{\ell}}),$$ with the choices of
$\bar\beta$ and $\xi$ as in Theorem \ref{thm:isotopyPqlBH}. Recall
that $\omega\bar\beta\in\F_{q^2}\setminus\F_q$, $\bar\beta$ is a
nonsquare in $\F_{q^{2\ell}}$, $\omega^2=\sigma\in\F_q$ and
$\bar\beta^{q+1}=\frac 1 \sigma$.

Let $\rho=2\omega\bar\beta$ and note that $\bar\phi^{-1}(\rho
x)=\psi(x)$, i.e. $\bar\phi^{-1}\circ t_\rho=\psi$, where
$t_\rho(x)=\rho x$.

Since $q\equiv 1(mod\,4)$ and $\omega^{q-1}=-1$, we have that
$\omega$ is a nonsquare in $\F_{q^2}$, and hence
$\rho=2\omega\bar\beta$ is a square in $\F_{q^2}$. Let
$b\in\F_{q^2}$ such that $b^2=\rho$ and let
$H(x)=\bar\phi^{\,-1}(bx)$, i.e. $H=\bar\phi^{\,-1}\circ t_b$
is an invertible $\F_p$--linear map of $\F_{q^{2\ell}}$. Then, by
(\ref{form:mapT}), we get $$HS_1\overline H=(\bar\phi^{\,-1}\circ
t_b)S_1(t_b\circ\phi^{\,-1}).$$

Since the elements of $S_1$ are $\F_{q^2}$--linear maps of
$\F_{q^{2\ell}}$ and $b\in\F_{q^2}$ we have $$HS_1\overline
H=(\bar\phi^{\,-1}\circ t_{b^2})S_1\phi^{\,-1}=(\bar\phi^{\,-1}\circ
t_{\rho})S_1\phi^{-1}=\psi S_1\phi^{-1}=S_2.$$

This proves the theorem.
\end{proof}

Finally, we can prove

\begin{theorem}\label{thm:nonstrongisotopy}
If $q\equiv -1(mod\,4)$, then the commutative presemifields
$P(q,\ell)$ and $\overline{B}(q,\ell,2,\bar\beta)$ of Corollary
\ref{cor:isotopy} are not strongly istopic.
\end{theorem}

\begin{proof}
By way of contradiction, suppose that the two involved presemifields
are strongly isotopic. Then by Corollary
\ref{cor:IsotopyComSimpl-spreads}, there exists an invertible
$\F_p$--linear map $H$ of $\F_{q^{2\ell}}$, $q=p^h$, such that
$HS_1\overline H=S_2$, where $S_1$ and $S_2$ are the spread sets
associated with $\S_1^{t*}$ and $\S_2^{t*}$, respectively. In
particular
$$
S_1=\{\varphi_{A,B}:\ x\mapsto Ax+B^{q^2}\eta x^{q^2}+B\eta
x^{q^{2\ell-2}}|\ y=A+(B^{q^2}+B)\eta,\,y\in\F_{q^{2\ell}}\}.$$ By
Theorem \ref{thm:isotopyPqlBH}, $\psi S_1\phi^{-1}=S_2$, hence
$\psi^{-1}HS_1\overline H\phi=S_1$, where
$$\psi^{-1}:x\mapsto \frac
12(\omega\bar\beta^qx+x^{q^\ell})\quad\quad\phi:x\mapsto
x-\omega\bar\beta^q x^{q^{\ell}}$$ and $\psi^{-1}=\frac 12
\omega\bar\beta^q\,\overline\phi$. It follows that
\begin{equation}\label{form:semilin}
\delta G S_1\bar G=S_1,
\end{equation} where $\delta=\frac 12 \omega\bar\beta^q\in\F_{q^2}$ and $G=\overline\phi
H$. Since the elements of $S_1$ are $\F_{q^2}$--linear maps of
$\F_{q^{2\ell}}$, by Theorem \ref{thm:semilinearity} and Proposition
\ref{prop:IsotopicSemifAndSpreadSets}, we have that $G$ is an
invertible $\F_{q^2}$--semilinear map of $\F_{q^{2\ell}}$, with
companion automorphism $\sigma=p^e$.

Let $$G(x)=
\sum_{i=0}^{\ell-1}a_ix^{p^{2hi+e}}=\sum_{i=0}^{\ell-1}a_i x^{\sigma q^{2i}},$$ then
$$\overline{G}(x)=\sum_{i=0}^{\ell-1}a_{i}^{p^{2\ell h-2hi-e}} x^{p^{2\ell h-2hi-e}}=
\sum_{i=0}^{\ell-1}a_i^{\sigma^{-1}q^{2\ell-2i}}x^{\sigma^{-1}q^{2\ell-2i}}.$$

By (\ref{form:semilin}), the map $\delta
(G\circ\varphi_{A,0}\circ\bar G)$ belongs to $S_1$ for each
$A\in\F_{q^\ell}$. Then there exist $A',B'\in\F_{q^\ell}$ such that
$\delta (G(A(\bar G(x))))=\varphi_{A',B'}(x)$ for each
$x\in\F_{q^{2\ell}}$.

Since $$\delta (G(A(\bar
G(x))))=\delta\Big(\sum_{j=0}^{\ell-1}\sum_{i=0}^{\ell-1}A^{\sigma
q^{2j}}a_ja_i^{q^{2(\ell-i+j)}}x^{q^{2(\ell-i+j)}}\Big)=A'x+B'^{q^2}\eta
x^{q^2}+B'\eta x^{q^{2\ell-2}},$$ reducing the above polynomial
identity modulo $x^{q^{2\ell}}-x$ and by comparing the
coefficients of first degree, we get
$$\delta(A^\sigma a_0^2+A^{\sigma q^2}a_1^2+\dots+A^{\sigma
q^{2\ell-2}}a_{\ell-1}^2)=A'\in\F_{q^\ell}$$ for each
$A\in\F_{q^\ell}$, i.e.
$$A^\sigma (\delta a_0^2-\delta^qa_0^{2q^\ell})+A^{\sigma q^2}(\delta a_1^2-\delta^qa_1^{2q^\ell})+\dots+A^{\sigma
q^{2\ell-2}}(\delta a_{\ell-1}^2-\delta^qa_{\ell-1}^{2q^\ell})=0$$
for each $A\in\F_{q^\ell}$. This is equivalent to
$$(\bar\beta^qa_0^2+\bar\beta a_0^{2q^\ell})x+(\bar\beta^qa_1^2+\bar\beta
a_1^{2q^\ell})x^{q^2}+\dots(\bar\beta^qa_{\ell-1}^2+\bar\beta
a_{\ell-1}^{2q^\ell})x^{q^{2\ell-2}}=0
$$ for each $x\in\F_{q^\ell}$. Reducing the above polynomial
identity over $\F_{q^\ell}$ modulo $x^{q^\ell}-x$, we get
$$\bar\beta^qa_i^2+\bar\beta a_i^{2q^\ell}=0$$ for each
$i\in\{0,1,\dots,\ell-1\}$. If $a_i\ne 0$, then $a_i$ is a solution of
$$x^{2q^\ell-2}=-\bar\beta^{q-1}.$$ However, when $q\equiv -1(mod\,4)$, the
last equation admits no solution in $\F_{q^{2\ell}}$. Hence the
unique $\F_{q^2}$--semilinear map satisfying (\ref{form:semilin}) is
the zero one, a contradiction.
\end{proof}

\bigskip

\bigskip

\noindent G. Marino and O. Polverino \\
Dip. di Matematica\\
Seconda Universit\`a degli Studi di Napoli\\
I--\,81100 Caserta, Italy\\
{\em giuseppe.marino@unina2.it}, {\em olga.polverino@unina2.it}

\end{document}